\documentclass[12pt]{amsart}
\usepackage{amssymb,amsmath,amsthm,latexsym}
\usepackage{mathrsfs}
\usepackage{a4wide}
\usepackage[usenames,dvipsnames]{color}
\usepackage{euscript}
\usepackage{graphicx}
\usepackage{mdwlist}
\usepackage{mathtools,dsfont}

\newtheorem{theorem}{Theorem}[section]
\newtheorem*{theorema}{Theorem A}
\newtheorem*{theoremb}{Theorem B}

\newtheorem{lemma}[theorem]{Lemma}

\newtheorem{proposition}[theorem]{Proposition}
\theoremstyle{remark}
\newtheorem{remark}[theorem]{Remark}
\newtheorem*{remark*}{Remark}
\theoremstyle{definition}
\newtheorem{definition}[theorem]{Definition}

\newtheorem{example}[theorem]{Example}

\numberwithin{equation}{section}
\makeatother

\newcommand{\n}[1]{\|#1\|}
\newcommand{\nn}[1]{{\vert\kern-0.25ex\vert\kern-0.25ex\vert #1 
    \vert\kern-0.25ex\vert\kern-0.25ex\vert}}
\newcommand{\abs}[1]{\vert #1\vert}
\newcommand{\spn}{\mathrm{span}}

\newcommand{\supp}{\mathrm{supp}}

\renewcommand{\leq}{\leqslant}
\renewcommand{\geq}{\geqslant}
\newcommand{\bb}{\boldsymbol}
\newcommand{\ccap}{\scalebox{0.95}{$\bigcap$}}

\newcommand{\R}{\mathbb{R}}
\newcommand{\N}{\mathbb{N}}
\newcommand{\Q}{\mathbb{Q}}

\newcommand{\e}{\varepsilon}

\newcommand{\ba}{\mathsf{ba}}
\newcommand{\PP}{\mathcal{P}}
\newcommand{\dd}{\mathrm{d}}
\newcommand{\ww}{\hat}
\newcommand{\oo}{\overline}
\newcommand{\w}{\widetilde}

\newcounter{smallromans}

\newenvironment{romanenumerate}
{\begin{list}{{\normalfont\textrm{(\roman{smallromans})}}}%
    {\usecounter{smallromans}\setlength{\itemindent}{0cm}%
      \setlength{\leftmargin}{5.5ex}\setlength{\labelwidth}{5.5ex}%
      \setlength{\topsep}{0.2ex}\setlength{\partopsep}{0ex}%
      \setlength{\itemsep}{0.2ex}}}%
  {\end{list}}

\newcounter{smallalphs}

  {\end{list}}

\DeclareFontFamily{U}{mathx}{\hyphenchar\font45}
\DeclareFontShape{U}{mathx}{m}{n}{
      <5> <6> <7> <8> <9> <10>
      <10.95> <12> <14.4> <17.28> <20.74> <24.88>
      mathx10
      }{}
\DeclareSymbolFont{mathx}{U}{mathx}{m}{n}
\DeclareFontSubstitution{U}{mathx}{m}{n}
\DeclareMathAccent{\widecheck}{0}{mathx}{"71}
\DeclareMathAccent{\wideparen}{0}{mathx}{"75}

\begin{document}
\date{February 17, 2016}
\title[Uncountable sets of unit vectors]{Uncountable sets of unit vectors\\
that are separated by more than 1}
\author[T. Kania]{Tomasz Kania}
\address{Mathematics Institute,
University of Warwick,
Gibbet Hill Rd, 
Coventry, CV4 7AL, 
England}
\email{tomasz.marcin.kania@gmail.com}

\author[T. Kochanek]{Tomasz Kochanek}
\address{Institute of Mathematics, Polish Academy of Sciences, \'Sniadeckich 8, 00-656 Warsaw, Poland {\rm and} Institute of Mathematics, University of Warsaw, Banacha~2, 02-097 Warsaw, Poland}
\email{tkoch@impan.pl}
\thanks{The research of the second-named author was supported by the Polish Ministry of Science and Higher Education in the years 2013-14, under Project No.~IP2012011072. The first-named author has also received funding from the European Research Council / ERC Grant Agreement No.~291497.}

\begin{abstract}Let $X$ be a Banach space. We study the circumstances under which there exists an uncountable set $\EuScript A\subset X$ of unit vectors such that $\|x-y\|>1$ for distinct $x,y\in \EuScript A$. We prove that such a set exists if $X$ is quasi-reflexive and non-separable; if $X$ is additionally super-reflexive then one can have $\|x-y\|\geqslant 1+\varepsilon$ for some $\varepsilon>0$ that depends only on $X$. If $K$ is a non-metrisable compact, Hausdorff space, then the unit sphere of $X=C(K)$ also contains such a subset; if moreover $K$ is perfectly normal, then one can find such a set with cardinality equal to the density of $X$; this solves a~problem left open by S.~K.~Mercourakis and G.~Vassiliadis. \end{abstract}


\maketitle

\section{Introduction} The study of distances between unit vectors in Banach spaces has a~long history that originates perhaps with the classical Riesz lemma  (\cite{riesz}), whose centennial is to be celebrated soon. There are two closely related problems in the geometry of Banach spaces that are often considered---given a Banach space $X$, what are the possible cardinalities of equilateral sets in $X$ (a set in a metric space is \emph{equilateral} if all of its members are equidistant from each other) and what are the possible cardinalities of sets of unit vectors in $X$ that are apart by a certain distance (typically greater than one)?\medskip

One of the early results concerning equilateral sets is due to Petty who studied distances between unit vectors in finite-dimensional Banach spaces (\cite{petty}). It is therefore a natural question of whether every infinite-dimensional Banach space contains an infinite equilateral set. This is the case for spaces of large enough density as observed by Terenzi (\cite{terenzi}) and one can always re-norm the space in such a way that the new norm has arbitrarily small Banach--Mazur distance from the old one and there exists an infinite equilateral set with respect to the new norm (\cite{mv1, swanepoel}). Hypotheses such as containing a copy of $c_0$ or having a uniformly smooth norm are also sufficient for the existence of an infinite equilateral set (\cite{freeman, mv1}). In general, however, this is not the case (\cite{gm2, terenzi1}).\medskip

Kottman (\cite{kottman}) provided a powerful extension of the Riesz lemma. According to his result, the unit sphere of every infinite-dimensional Banach space contains an infinite set of vectors that are in distance greater than one from each other. (We shall call such sets \emph{$(1+)$-separated}.) In \cite[pp.~7--8]{diestel}, one may find a very elegant and surprisingly short proof of Kottman's theorem whose authorship is explained by Diestel as follows: {\it `We were shown this proof by Bob Huff who blames Tom Starbird for its simplicity.'} A yet another proof can be found in \cite{gm}. Elton and Odell (\cite{eltonodell}) gave a beautiful, Ramsey-theoretic proof of the following theorem: for each infinite-dimensional Banach space $X$ there exists $\varepsilon >0$ and an infinite $(1+\varepsilon)$-separated subset of the unit sphere of $X$. (Given $\delta>0$, a set is \emph{$\delta$-separated} if all of its member are in the distance at least $\delta$ from each other.) Kryczka and Prus (\cite{kryczkaprus}) proved that if $X$ is non-reflexive then the unit sphere of $X$ contains an infinite, $\sqrt[5]{4}$-separated subset.\medskip

All the above-mentioned constructions of separated subsets of the unit sphere yield actually sequences (countable sets). Of course, in the case of separable Banach spaces this is the best one may expect as all discrete subsets of separable metric spaces are countable. What happens beyond the separable case? The aim of this is paper is to answer this question at least partially. \medskip

Our starting point is the observation by Elton and Odell (\cite[Remark on p.~109]{eltonodell}) who noticed that for all $\varepsilon >0$ each $(1+\varepsilon)$-separated subset of the unit sphere of $c_0(\omega_1)$ is countable. For the reader's convenience, we present the proof in Section~\ref{prelim}. Curiously enough, the unit sphere of $c_0(\omega_1)$ contains an uncountable (1+)-separated subset (it was mentioned without a proof in \cite[p.~12]{gm}). The proof is so simple that we include it here. \medskip

Towards a contradiction, assume that each (1+)-separated set of unit vectors in $c_0(\omega_1)$ is countable and take one, $\{f_n\colon n\in \mathbb{N}\}$ say, that is maximal with respect to inclusion and whose each member assumes the value 1. The union $D$ of the supports of all $f_n$'s ($n\in \mathbb{N}$) is countable, thus $D=\{\alpha_1, \alpha_2, \ldots \}$ for some $\alpha_k<\omega_1$ ($k\in \mathbb{N}$). Pick $\alpha_0\notin D$ and set $f(\alpha_0)=1$, $f(\alpha_k)=-\frac{1}{k}$ ($k\in \mathbb{N}$) and $f(\alpha)=0$ otherwise. Then $f\in c_0(\omega_1)$, $\|f\|=1$ and $\|f-f_n\|>1$ for each $n$ which contradicts the maximality of $\{f_n\colon n\in \mathbb{N}\}$.\medskip

We prove that for a substantial class of non-separable Banach spaces one can find an uncountable (1+)-separated subset of the unit sphere. 

\begin{theorema}Let $X$ be a non-separable Banach space. Then\smallskip
\begin{romanenumerate}
\item if $X$ is quasi-reflexive, the unit sphere of $X$ contains an uncountable {\rm (1+)}-separated subset;\smallskip
\item if $X$ is super-reflexive, for each regular cardinal number $\kappa$ that does not exceed the density of $X$ there exist  $\varepsilon>0$ and a~${\rm (1+}\varepsilon{\rm )}$-separated subset of the unit sphere of $X$ that has cardinality $\kappa$.\newline \indent In particular, the unit sphere of a non-separable super-reflexive space contains an uncountable ${(\rm 1+}\varepsilon{\rm )}$-separated subset for some $\varepsilon>0$.\smallskip
\item if $X$ is a~WLD Banach space of density bigger than continuum, the unit sphere of $X^*$ contains an uncountable {\rm (1+)}-separated subset.\end{romanenumerate}
\end{theorema}\noindent
A~Banach space $X$ is called \emph{quasi-reflexive} if $X^{**} / X$ is finite-dimensional; in particular, reflexive spaces are quasi-reflexive, \emph{a fortiori}. A Banach space $X$ is \emph{super-reflexive} if each ultrapower of $X$ is reflexive. The first part of this result is Theorem~\ref{gwiezdnyptak}, the second clause is Theorem~\ref{czupakabra}, whereas the third one is Theorem~\ref{long_PRI}. We borrowed the very idea of the proof of Theorem~\ref{gwiezdnyptak} from Starbird's proof of Kottman's theorem (\cite[p.~7--8]{diestel}) which is very finitary in its nature---it deals with norm estimates of certain linear combinations. We employ the transfinite induction in order to extract an uncountable (1+)-separated subset, however a good number of obstacles must be circumvented as, for instance, our `generalised linear combinations' are no longer finite.\medskip

In Section~\ref{finalsection}, we specialise to the class of Banach spaces of continuous functions on locally compact Hausdorff spaces. Propositions~\ref{hs}, \ref{td}, and \ref{maharam} had been obtained independently of Mercourakis and Vassiliadis (\cite{mv}) at the beginning of 2014---we take this opportunity to acknowledge their priority as these results overlap with \cite[Theorem 2]{mv}. The proofs we provide are in most cases different and it is for the reader to decide whether they are more elementary or not. \medskip

We also solve a problem left open by Mercourakis and Vassiliadis (\cite[Question 2.7(1)b]{mv}) who asked whether the unit sphere of $C(K)$, where $K$ is non-metrisable, contains an uncountable (1+)-separated subset. We prove actually a stronger result relating the cardinality of such sets to the density of $C(K)$, that is the minimal cardinality of a dense subset of $C(K)$.

\begin{theoremb}Let $K$ be a non-metrisable, compact Hausdorff space. Then, the unit sphere of $C(K)$ contains an uncountable, {\rm (1+)}-separated subset. Furthermore, \smallskip
\begin{romanenumerate}
\item if $K$ is not perfectly normal, then the unit sphere of $C(K)$ contains an uncountable 2-equilateral set;\smallskip
\item if $K$ is perfectly normal, then the unit sphere of $C(K)$ contains a {\rm (1+)}-separated subset of cardinality equal to the density of $C(K)$.\end{romanenumerate}\end{theoremb}
\noindent
The proof of Theorem B is a conjuction of Proposition~\ref{lemmamv} with Theorem~\ref{weight}.

\section{Preliminaries and auxiliary results}\label{prelim}

We work with real Banach spaces however most of the results can be easily generalised to the case of complex scalars. We adopt the von Neumann definition of an ordinal number and we identify cardinal numbers with initial ordinal numbers. By $\mathfrak{c}$, we denote the cardinality of continuum, whereas $\omega_1, \omega_2, \ldots$ are the first and, respectively, the second etc.~uncountable ordinal number. We follow the fashion that is very alive in certain circles to keep the symbols $\omega, \omega_1, \omega_2, \ldots$ etc. for infinite cardinal numbers (thus, $\omega=\aleph_0$, $\omega_1 = \aleph_1$, etc.). Given a cardinal number $\lambda$, we denote by ${\rm cf}\,\lambda$ the \emph{cofinality} of $\lambda$, that is, the smallest cardinal number $\kappa$ such that $\lambda$ can be written as a union of $\kappa$ many sets each of cardinality less than $\lambda$. A cardinal number $\lambda$ is \emph{regular} if $\lambda={\rm cf}\, \lambda$ and \emph{singular} otherwise. For the brevity of notation, we introduce the following cardinal functions.\medskip

Let $X$ be a Banach space. Denote by $S_X$ the unit sphere of $X$. We define

$$\begin{array}{lcl}\mathsf k (X) &=& \sup \{|\EuScript A|\colon \EuScript A\subseteq S_X, \, \|x-y\|>1\text{ for all distinct }x,y\in \EuScript A\}\text{, and}\\
\mathsf{eo} (X) &=& \sup \{|\EuScript A|\colon \EuScript A\subseteq S_X, \, \|x-y\|\geqslant 1+\varepsilon\text{ for some }\varepsilon>0\text{ and all distinct }x,y\in \EuScript A\}.\end{array}$$The attentive reader will note that the names of the just-defined cardinal functions refer to the aforementioned theorems of Kottman and Elton--Odell, respectively.\medskip

Given a Banach space $X$, denote by $\mathsf{d}(X)$ the minimal cardinality of a dense subset of $X$. Thus, we have the following immediate inequality:\medskip
$$ \mathsf{eo} (X) \leqslant \mathsf k(X) \leqslant \mathsf d(X).$$

Let $\Gamma$ be a set. Then the vector space $c_0(\Gamma)$ that consists of all scalar-valued functions $f$ on $\Gamma$ such that the set $\{\gamma\in \Gamma\colon |f(\gamma)|\geqslant \varepsilon\}$ is finite for all $\varepsilon > 0$ is a Banach space when endowed with the supremum norm. We regard $c_0(\omega_1)$ as a paradigm example of a Banach space $X$ for which the numbers $\mathsf k (X)$ and $\mathsf{eo}(X)$ differ. We have proved already in the introduction that $\mathsf{k}(c_0(\omega_1))=\omega_1$ and we mentioned that it is a result of Elton and Odell (\cite[p.~109]{eltonodell}) that $\mathsf{eo}(c_0(\omega_1))=\omega$. Let us record this formally here; we present also a detailed proof.

\begin{proposition}$\omega = \mathsf{eo}(c_0(\omega_1))<\mathsf{k}(c_0(\omega_1))=\omega_1.$ \end{proposition}

\begin{proof}It is enough to prove that $\mathsf{eo}(c_0(\omega_1))$ is countable. \medskip

Assume that for some $\e>0$ there exists an uncountable $(1+\e)$-separated set $\EuScript A\subset S_{c_0(\omega_1)}$. For each $x\in \EuScript A$ define $$F_x=\Bigl\{\alpha\in [0,\omega_1)\colon\abs{x(\alpha)}>\frac{\e}{2}\Bigr\}$$which is, of course, a~finite set. According to the $\Delta$-system lemma (\cite[Theorem 1.5]{kunen}), there exists an uncountable subfamily $\EuScript B$ of $\EuScript A$ and a~finite set $\Delta\subset\omega_1$ such that
$$
F_x\cap F_y=\Delta\qquad(x,y\in \EuScript B,\, x\not=y).
$$
Consider any two distinct members $x$ and $y$ of $\EuScript B$. If $\alpha\in\omega_1\setminus\Delta$, then at least one of the coordinates $x(\alpha)$, $y(\alpha)$ lies inside the interval $[-\frac{\e}{2},\frac{\e}{2}]$, hence $\abs{x(\alpha)-y(\alpha)}\leqslant 1+\frac{\e}{2}$. Therefore, the validity of $\n{x-y}\geq 1+\e$ may only be witnessed by coordinates from $\Delta$. Define a~`pattern function' $p\colon \EuScript B\to\{-1,1\}^\Delta$ by 
$$
p(x)(\alpha)=\left\{\begin{array}{rl}
+1 & \mbox{if }x(\alpha)>\frac{\e}{2},\\
-1 & \mbox{if }x(\alpha)<-\frac{\e}{2}
\end{array}\right.\qquad (x\in \EuScript B, \alpha\in \Delta).
$$
There must exist two distinct vectors $x,y\in \EuScript B$ with $p(x)=p(y)$. Consequently, $$\abs{x(\alpha)-y(\alpha)}\leq 1-\frac{\e}{2}\qquad (\alpha\in\Delta).$$ Therefore, $\n{x-y}\leq 1+\frac{\e}{2}$ and we arrive at a~contradiction.\end{proof}

We proceed now to general results concerning $\mathsf{k}(X)$ and $\mathsf{eo}(X)$. The following lemma is a well-known observation; it can be found, \emph{e.g.}, in \cite[Lemma 6]{martini}. We provide here a simple proof for the sake of completeness.

\begin{lemma}\label{convex}Let $X$ be a normed space and let $x,y\in X$ be non-zero vectors such that $\|x\|, \|y\|\leqslant 1$ and $\|x-y\|\geqslant 1$. Then 
$$\left\| \frac{x}{\|x\|} - \frac{y}{\|y\|} \right\| \geqslant\|x-y\|.$$\end{lemma}
\begin{proof}Assume with no loss of generality that $\n{x}\geqslant\n{y}$ and consider the function given by $g(t) = \|x - ty\|$ ($t\in \mathbb{R}$); it is easy to verify that this is a~convex function. We have
$$
\left\| \frac{x}{\|x\|} - \frac{y}{\|y\|} \right\| =\frac{1}{\n{x}}\left\|x-\frac{\n{x}}{\n{y}}y\right\|\geqslant g\Bigl(\frac{\n{x}}{\n{y}}\Bigr).
$$
Since $g(0)=\n{x}\leqslant 1$ and $g(1)=\n{x-y}\geqslant 1$, the convexity of $g$ implies that $g(\n{x}/\n{y})$ must be at least equal to $g(1)$.\end{proof}

\begin{proposition}\label{lifting}Suppose that $X$ and $Y$ are Banach spaces and $Y$ is isometric to a quotient of $X$. If for some $\varepsilon>0$ there exists a $(1+\varepsilon)$-separated subset $\{y_i\colon i\in I\}$ of unit vectors in $Y$, then for all $\delta\in (0,\varepsilon)$ there exists a $(1+\delta)$-separated subset of the unit sphere of $X$ with cardinality $|I|$. Consequently, $\mathsf{eo}(Y)\leqslant \mathsf{eo}(X).$ \end{proposition}
\noindent \emph{Proof.} Let $\pi\colon X\to Y$ be the quotient map and choose $\eta\in (0, \frac{\varepsilon-\delta}{1+\delta})$. For each $i\in I$ let $u_i\in X$ so that $\pi(u_i) = y_i$ and $\|u_i\|\leqslant 1+\eta$. Set $x_i = u_i / \|u_i\|$ ($i\in I$). Then using Lemma~\ref{convex}, we have
\[\pushQED{\qed}\|x_i - x_j\| = \left\| \frac{u_i}{\|u_i\|} - \frac{u_j}{\|u_j\|}\right\|\geqslant \frac{\|\pi(u_i-u_j)\|}{1+\eta}= \frac{\|y_i - y_j\|}{1+\eta} > 1+\delta.\qedhere
\popQED
\]

\begin{example}Let $J\!L$ be the Johnson--Lindenstrauss space (\cite{jl}). Denote by $M$ the canonical copy of $c_0$ in $J\! L$. The quotient space $J\!L/M$ is isometric to $\ell_2(\mathfrak{c})$ hence $\mathsf{eo}(J\!L)\geqslant \mathfrak{c}$. On the other hand, the cardinality of $J\! L$ itself is $\mathfrak{c}$, so $\mathsf{eo}(J\!L)= \mathfrak{c}$. \end{example}
 
\begin{proposition}\label{reflexive}Let $X$ be a Banach space and let $M\subseteq X$ be a reflexive subspace. Set $Y = X/M$. Then $\mathsf{k}(Y)\leqslant \mathsf{k}(X).$\end{proposition}
\begin{proof}Since $M$ is reflexive, for each $y\in Y$ there is $x\in X$ such that $\pi(x)=y$ and $\|x\|=\|y\|$. Given a $(1+)$-separated family of unit vectors in $Y$, for each member of this family choose an appropriate lifting and use the fact that $\pi$ does not increase the norm.\end{proof}

\section{Reflexive spaces}\label{quasireflexivesection}

\subsection{Quasi-reflexive spaces}
We start this section with giving a generalisation of Starbird's argument, mentioned in the Introduction, to the non-separable setting. According to a result by Civin and Yood (\cite[Theorem 4.6]{cy}), every non-separable quasi-reflexive space contains a~non-separable reflexive subspace, whence the next result implies assertion~(i) of Theorem~A.
\begin{theorem}\label{gwiezdnyptak}
Let $X$ be a Banach space containing a non-separable reflexive subspace. Then $\mathsf{k}(X)$ is uncountable.

\end{theorem}
Before going into action, we need some preparations.\medskip 

Let $y^\ast\in X^\ast$ and let $(x_n^\ast)_{n=1}^\infty$ be a bounded sequence in $X^\ast$. Denote by $\mathrm{span}_\Q\{x_n^\ast\colon n\in\N\}$ the set of all (finite) linear combinations of $x_n^\ast$'s with rational coefficients. We say that $y^\ast$ is a {\it generalised combination} of $x_n^\ast$'s ($n\in \mathbb{N}$) if for some enumeration $\{z_1^\ast,z_2^\ast,\ldots\}$ of the set $\mathrm{span}_\Q\{x_n^\ast\colon n\in\N\}\cap B_{X^\ast}$ there exists a~measure $\mu\in\ba(\PP\N)$ (here $\ba(\PP\N)$ denotes the family of all scalar-valued, bounded, finitely additive measures on the $\sigma$-algebra $\PP\N$ of all subsets of $\N$; we identify this space with $\ell_\infty^*$) such that 
\begin{equation}\label{mu}
\langle y^\ast, x\rangle=\int\limits_\N \langle z_n^\ast, x\rangle\,\mu(\dd n)\qquad(x\in X).
\end{equation}
It is a well-known result from linear algebra that given finitely many linear functionals $y^\ast,x_1^\ast,\ldots,x_N^\ast$ so that $y^\ast$ vanishes whenever all $x_i^\ast$'s vanish, the functional $y^\ast$ must be a~linear combination of $x_i^\ast$'s (see, {\it e.g.}, \cite[Lemma~3.9]{rudin}). The following lemma yields an infinite version of this statement.  We shall use the fact that for a quasi-reflexive space $X$ every total subspace $M$ of $X^\ast$ is norming (see, \emph{e.g.}, \cite{petunin, singer}), that is, for some positive constant $c$ we have $$\sup\bigl\{\abs{\langle x^\ast, x\rangle}\colon x^\ast\in M\cap B_{X^\ast}\bigr\}\geq c\n{x}\qquad(x\in X).$$ (In fact, this property characterises quasi-reflexive spaces, as was shown by Davis and Lindenstrauss in \cite{davis_lindenstrauss}.)
\begin{lemma}\label{gc}
Let $X$ be a quasi-reflexive Banach space. Suppose that $y^\ast\in X^\ast$ and let $(x_n^\ast)_{n=1}^\infty$ be a~bounded sequence in $X^\ast$ such that $$\bigcap_{n=1}^\infty\mathrm{ker}(x_n^\ast)\subseteq\mathrm{ker}(y^\ast).$$ Then $y^\ast$ is a generalised combination of $x_n^\ast$'s {\rm (}$n\in \mathbb{N}${\rm )}.
\end{lemma}
\begin{proof}
Set $N=\ccap_{n=1}^\infty\mathrm{ker}(x_n^\ast)$ and let $z_1^\ast,z_2^\ast,\ldots$ be any enumeration of all vectors from $\spn_\Q\{x_n^\ast\}_{n=1}^\infty\cap B_{X^\ast}$. Define a linear map $\Lambda\colon X\to\ell_\infty$ by $$\Lambda(x)=(\langle z_n^\ast, x\rangle)_{n=1}^\infty\qquad (x\in X)$$ and observe that, by our assumption, we have $\langle y^\ast, x\rangle =\langle y^\ast, x^\prime\rangle$ whenever $\Lambda(x)=\Lambda(x^\prime)$. This makes it possible to define a linear map $f\colon\ell_\infty\supset\Lambda(X)\to\R$ by $\langle f, \Lambda(x)\rangle =\langle y^\ast, x\rangle$. \medskip

Since quotients of a quasi-reflexive space are quasi-reflexive too (\cite[Corollary 4.2]{cy}), the space $X/N$ is quasi-reflexive. For any functional $x^\ast\in X^\ast$ that vanishes on $N$ define $\ww{x}^\ast\in (X/N)^\ast$ by $\langle \ww{x}^\ast, x+N\rangle =\langle x^\ast, x\rangle$; of course, $\n{x^\ast}=\n{\ww{x}^\ast}$ (\emph{cf.} \cite[Theorem~4.9(b)]{rudin}). As the set $\{\ww{x}_n^\ast\}_{n=1}^\infty$ is total in $(X/N)^\ast$, the subspace $\oo\spn\{\ww{x}_n^\ast\}_{n=1}^\infty$ is norming, whence so is $\spn_\Q\{\ww{x}_n^\ast\}_{n=1}^\infty$. So, let $c>0$ be so that 
$$
\sup_{n\in \mathbb{N}}\abs{\langle \ww{z}_n^\ast, x+N\rangle }\geq c\n{x+N}_{X/N}\qquad(x\in X).
$$
For any $x\in X$ pick $y_x\in N$ for which $\n{x+y_x}\leq 2\n{x+N}_{X/N}$ and note that
\begin{equation*}\begin{aligned}
\n{\Lambda(x)}_{\ell_\infty}\,&=&\,\sup_{n\in \mathbb{N}}\abs{\langle z_n^\ast,x\rangle }\,\\
&=& \sup_{n\in \mathbb{N}}\abs{\langle \ww{z}_n^\ast, x+N\rangle }\,\\
&\geq& c\n{x+N}_{X/N}\,\\
&\geq& \frac{1}{2}\cdot c\n{x+y_x}.
\end{aligned}\end{equation*}
Therefore,
$$
\abs{\langle f, \Lambda(x)\rangle}=\abs{\langle y^\ast, x+y_x\rangle }\leq\n{y^\ast}\cdot\n{x+y_x}\leq\frac{2}{c}\n{y^\ast}\cdot\n{\Lambda(x)}_{\ell_\infty}
$$
which shows that $f$ is a continuous functional. So, let $F\in\ell_\infty^\ast\cong\ba(\PP\N)$ be any norm-preserving extension of $f$ and let $\nu\in\ba(\PP\N)$ be the measure corresponding to $F$. Then, for every $x\in X$ we have
$$
\langle y^\ast, x\rangle =\langle F, \Lambda(x)\rangle=\langle F, (z_n^\ast)_{n=1}^\infty\rangle=\int\limits_\N \langle z_n^\ast, x\rangle\,\nu(\dd n),
$$
which completes the proof. 
\end{proof}
\begin{remark}\label{rem_est}
Note that the above proof gives the upper estimate
\begin{equation}\label{mu_est}
\n{\mu}\leq\frac{2}{c}\n{y^\ast}
\end{equation}
for the variation norm of the measure $\mu$ representing $y^\ast$ via formula \eqref{mu}. Here, $c$ depends only on the given sequence $(x_n^\ast)_{n=1}^\infty$. Note also that Lemma~\ref{gc}, and especially inequality \eqref{mu_est}, is related to a result of Ostrovskii (\cite{ostrovskii}) saying that for any quasi-reflexive space $X$ and any absolutely convex set $A\subset X^\ast$ (in particular, any subspace) we have $$\oo{A}^{w^\ast}=\bigcup_{n=1}^\infty\oo{A\cap nB_{X^\ast}}^{w^\ast},$$ the set on the right-hand side being called the {\it weak$^\ast$ derived set of} $A$. Inequality \eqref{mu_est} gives a~uniform bound on the norms of combinations of $z_n^\ast$'s that approximate any given functional being a~generalised combination of $z_n^\ast$'s. (For our purposes it is more natural to use generalised combinations rather than the weak$^\ast$ closure.)
\end{remark}
\begin{proof}[Proof of Theorem \ref{gwiezdnyptak}]
We may assume that $X$ itself is reflexive in which case we have $\mathsf{d}(X)=\mathsf{d}^{w^\ast}(X^\ast)$, where $\mathsf{d}^{w^\ast}(X^\ast)$ is the minimal cardinality of a~weak$^\ast$-dense subset of $X^\ast$, thereby $(X^\ast, w^\ast)$ is non-separable. 

We shall construct a sequence $(x_\alpha)_{\alpha<\omega_1}\subset X$ of unit vectors such that $\n{x_\alpha-x_\beta}>1$ whenever $0\leq\beta<\alpha<\omega_1$. Choose any unit vector $x_0$ in $X$ and pick $x_0^\ast\in S_{X^\ast}$ with $\langle x_0^\ast, x_0\rangle=1$. Now, fix any $\alpha<\omega_1$ and assume that we have already chosen vectors $(x_\beta)_{0\leq\beta<\alpha}\subset S_{X}$ and functionals $(x_\beta^\ast)_{0\leq\beta<\alpha}\subset S_{X^\ast}$ such that\medskip
\begin{romanenumerate}
\item $\langle x_\beta^\ast,x_\beta\rangle=1$ for every $\beta<\alpha$;\smallskip
\item $\n{x_\beta-x_\gamma}>1$ for all $0\leq\beta<\gamma<\alpha$;\smallskip
\item $\bigcap_{\gamma\not=\beta}\mathrm{ker}(x_\gamma^\ast)\not\subseteq\mathrm{ker}(x_\beta^\ast)$ for every $\beta<\alpha$.
\end{romanenumerate}\medskip

We \emph{claim} that there exists a vector $y\in X$ satisfying $\langle x_\beta^\ast, y\rangle<0$ for each $\beta<\alpha$. Indeed, by (iii), for every $\beta<\alpha$ we may find a unit vector $$z_\beta\in\bigcap_{\gamma\not=\beta}\mathrm{ker}(x_\gamma^\ast)$$ such that $\langle x_\beta^\ast, z_\beta\rangle<0$. Take any summable sequence $(c_\beta)_{\beta<\alpha}$ of positive numbers. Then, $y=\sum_{\beta<\alpha}c_\beta z_\beta$ has the desired property.\medskip

Set $N=\ccap_{\beta<\alpha}\mathrm{ker}(x_\beta^\ast)$. Since $N^\perp=\overline{\mathrm{span}}^{w\ast}\{x_\beta^\ast\}_{\beta<\alpha}$ and $(X^\ast,w^\ast)$ is non-separable, there exists a~non-zero vector $x\in N$. Let $c>0$ be the constant occurring in inequality \eqref{mu_est} corresponding to (any enumeration of) ${\rm span}_{\mathbb{Q}}\{x_\beta^*\colon \beta<\alpha\}\cap B_{X^*}$. Choose $K>0$ so large that $$\n{y+Kx}>\frac{3}{c}\n{y}$$ and define $$x_\alpha=\frac{y+Kx}{\n{y+Kx}}.$$
Let also $x_\alpha^\ast$ be any norm-one functional satisfying $\langle x_\alpha^\ast,x_\alpha\rangle=1$.\medskip

In order to verify condition (ii) for $\alpha+1$ in the place of $\alpha$, observe that for each $\beta<\alpha$ we have
$$
\n{x_\alpha-x_\beta}\geq\abs{\langle x_\beta^\ast, x_\alpha-x_\beta\rangle }=1-\frac{\langle x_\beta^\ast, y\rangle}{\n{y+Kx}}>1,
$$
as $\langle x_\beta^\ast,y\rangle$ is negative.\medskip

Regarding condition (iii), we start by showing that $x_\alpha^\ast$ is not a~generalised combination of $x_\beta^\ast$'s ($\beta<\alpha$). If it were, then for some enumeration $z_1^\ast,z_2^\ast,\ldots$ of $\spn_\Q\{x_\beta^\ast\}_{\beta<\alpha}\cap B_{X^\ast}$ and some measure $\mu\in\ba(\PP\N)$ we would have
$$\langle x_\alpha^\ast,z\rangle=\int\limits_\N \langle z_n^\ast, z\rangle\,\mu(\dd n)\qquad(z\in X).$$
Moreover, in view of Remark~\ref{rem_est}, we may assume that $\n{\mu}\leq 2/c$. Set $z=y+Kx$. For each $\e>0$ we may then find $N\in\N$ and scalars $a_1,\ldots,a_N$ so that $$\sum_{j=1}^N\abs{a_j}\leq\n{\mu}\leq\frac{2}{c}\quad\text{and}\quad \left|\langle x_\alpha^\ast, y+Kx\rangle-\sum_{j=1}^Na_j\langle z_j^\ast, y+Kx\rangle\right|<\e.
$$
However, the left-hand side of the latter inequality equals to
\begin{equation*}
\left|\n{y+Kx}-\sum_{j=1}^Na_j\langle z_j^\ast, y\rangle \right|\geq\n{y+Kx}-\n{y}\sum_{j=1}^N\abs{a_j}>\left(\frac{3}{c}-\n{\mu}\right)\n{y}\geq\frac{\n{y}}{c},
\end{equation*}
hence we arrive at a contradiction in the case where $\e<\n{y}/c$. Now, Lemma~\ref{gc} implies that
$$
\bigcap_{\beta<\alpha}\mathrm{ker}(x_\beta^\ast)\not\subseteq\mathrm{ker}(x_\alpha^\ast).
$$Finally, suppose that for some $\beta_0<\alpha$ we have 
$$
\mathrm{ker}(x_\alpha^\ast)\cap\bigcap_{\begin{smallmatrix} \beta<\alpha\\ \beta\not=\beta_0  \end{smallmatrix} }\mathrm{ker}(x_\beta^\ast)\subseteq\mathrm{ker}(x_{\beta_0}^\ast).
$$Appealing once again to Lemma~\ref{gc}, we infer that for some enumeration $\{w_1^\ast,w_2^\ast,\ldots\}$ of the set $\mathrm{span}_\Q\big\{x_\alpha^\ast\cup\{x_\beta^\ast\colon \beta<\alpha,\, \beta\not=\beta_0\}\big\}\cap B_{X^*}$ and for some measure $\nu\in\ba(\PP\N)$ we have 
$$
\langle x_{\beta_0}^\ast,z\rangle=\int\limits_\N \langle w_n^\ast, z\rangle\,\nu(\dd n)\qquad(z\in X).
$$
For each $n\in\N$ let $\theta_n$ be the coefficient of $x_\alpha^\ast$ in the linear combination $w_n^\ast$. Putting $z=x$ in the above equality, and using the fact that $x$ belongs to the kernels of all $x_\beta^\ast$ ($\beta<\alpha$), we obtain $\int_\N\theta_n\,\nu(\dd n)=0$. On the other hand, for every $z\in X$ we have $$\langle x_{\beta_0}^\ast,z\rangle=\langle x_\alpha^\ast,z\rangle\int\limits_\N\theta_n\,\nu(\dd n)+\int\limits_\N\langle\widetilde{w_n^\ast},z\rangle\,\nu(\dd n),$$
where $\widetilde{w_n^\ast}$ is the linear combination resulting from $w_n^\ast$ by truncating $\theta_n x_\alpha^\ast$. Since the first integral vanishes, we see that the kernel of $x_{\beta_0}^\ast$ contains $\ccap_{\beta<\alpha, \beta\not=\beta_0}\mathrm{ker}(x_\beta^\ast)$ which is false due to the induction hypothesis. Therefore, condition (iii) for $\alpha+1$ instead of $\alpha$ has been fully verified and the proof is complete.
\end{proof}

\subsection{Super-reflexive spaces}

Now, we shall show that Theorem~\ref{gwiezdnyptak} may be strengthened for super-reflexive spaces. We require a piece of terminology before explaining the result.\medskip

Let $\Lambda$ be a set and let $(X_\alpha)_{\alpha\in \Lambda}$ be a family of Banach spaces. Given $p\in [1,\infty)$, the \emph{$\ell_p$-sum} of $(X_\alpha)_{\alpha\in \Lambda}$ is the Banach space $Z=( \bigoplus_{\alpha\in \Lambda} X_\alpha)_{\ell_p}$ that consists of all tuples $x=(x_\alpha)_{\alpha\in \Lambda}$ such that $x_\alpha\in X_\alpha$ ($\alpha\in \Lambda$) and 
$\|x\| = ( \sum_{\alpha\in\Lambda}\|x_\alpha\|^p )^{1/p}<\infty.$
Given $x=(x_\alpha)_{\alpha\in \Lambda}\in Z$, the \emph{support} of $x$ is the set $ {\rm supp}\,x = \{\alpha\in \Lambda\colon x_\alpha\neq 0\}.$\medskip

Let $X$ be a non-separable Banach space and set $\lambda = \mathsf d(X)$. A family $(P_\alpha)_{\omega\leqslant \alpha\leqslant\lambda}$ of linear projections on $X$ is called \emph{a~projectional resultion of the identity in} $X$ (a \emph{PRI}, for short) if
\begin{romanenumerate}
\item $\|P_\alpha\|=1$ for all $\omega<\alpha\leqslant \lambda$,
\item $P_\omega=0$ and $P_\lambda = I_X$, the identity map on $X$,
\item $\mathsf d(P_{\alpha}(X))\leqslant |\alpha|$ for each $\alpha<\lambda$,
\item\label{lorded} $P_\alpha P_\beta = P_\beta P_\alpha = P_{\min \{\alpha, \beta\}}$ for all $\omega\leqslant \alpha<\beta\leqslant \lambda$,
\item \label{subspacespri}for every $x\in X$ the map $\alpha\mapsto P_\alpha x$ is continuous in the order-norm topology.\end{romanenumerate}\medskip

Condition \eqref{lorded} implies that the ranges of $P_\alpha$ ($\omega\leqslant \alpha\leqslant\lambda$) are well-ordered by inclusion. When talking about any PRI, we shall always assume without loss of generality that for different $\alpha,\beta$ the ranges of $P_\alpha$ and $P_\beta$ are different. We want also to emphasise that it follows from condition (v) that if $W\subset X$ is a~subspace with $\mathsf{d}(W)<\lambda$ then for some $\alpha<\lambda$ we have $W\subseteq P_\alpha(X)$. We refer to \cite[Chapter 13.2]{hajeketal} for more information concerning projectional resolutions of the identity.
\medskip

Let $X$ be a non-separable, reflexive space. By a result of Lindenstrauss (\cite{lin}), there exists a~PRI for $X$, say $(P_\alpha)_{\omega\leqslant \alpha\leqslant\lambda}$. Our idea relies on the observation due to Benyamini and Starbird (\cite[p.~139]{benyaministarbird}) who, building on work of James (\cite{james}), observed that if $X$ is super-reflexive, then for any $\e>0$ there exists $p\in (1,\infty)$ such that the operator

\begin{equation}\label{operator}T\colon X \;\longrightarrow\;\; \Biggr( \bigoplus_{\omega\leqslant \alpha<\lambda} (P_{\alpha+1}-P_\alpha)(X) \Biggr)_{\!\!\ell_p},\qquad Tx = \bigl(P_{\alpha+1}x - P_{\alpha}x\bigr)_{\omega\leqslant \alpha<\lambda}\qquad (x\in X)\end{equation}
has norm at most $2+\e$. This follows from James' theorem (see \cite[Theorem~4]{james}) which says that for each super-reflexive Banach space $X$ and any constants $0<c<1/(2K)$, $C>1$ there are exponents $1<q<p<\infty$ such that for every normalised basic sequence $(e_n)_{n=1}^\infty\subset X$ with basis constant $K$, and any scalars $(a_n)_{n=1}^N$, we have \begin{equation}\label{superjames}
c\cdot\!\left(\sum_{n=1}^N\abs{a_n}^p\right)^{\!\!\! 1/p}\leq\left\|\sum_{n=1}^Na_ne_n\right\| \leqslant C\!\cdot\!\left(\sum_{n=1}^N\abs{a_n}^q\right)^{\!\!\! 1/q}.
\end{equation}
Moreover, if $(e_n)_{n=1}^\infty$ is monotone, then $c$ may be taken to be arbitrarily close to $\frac{1}{2}$. Once $c<\frac{1}{2}$ is chosen (together with the corresponding value of $p$), the norm of $T$ equals at most $\frac{1}{c}$; moreover, $c$ may be chosen as close to $\frac{1}{2}$ as we wish, since any sequence of unit vectors picked from different blocks of the form $(P_{\beta+1}-P_\beta)(X)$ forms a~monotone basic sequence.
\begin{lemma}\label{gamma_min}
Let $X$ be super-reflexive space such that $\mathsf d(X)$ has uncountable cofinality and let $T$ be given by \eqref{operator}. Then there exists a~positive constant $\gamma_{\mathrm{min}}$ such that  for each closed subspace $Y\subseteq X$ with $\mathsf{d}(X/Y)<\mathsf{d}(X)$ we have $\|T|_Y\|\geqslant \gamma_{\mathrm{min}}$. 
\end{lemma}

\begin{proof}
Assume, on the contrary, that there is a~sequence of closed subspaces $Y_n\subset X$ ($n\in\N$) so that $\mathsf{d}(X/Y_n)<\mathsf{d}(X)$ and $\n{T\vert_{Y_n}}\to 0$ as $n\to \infty$. Let $Y = \bigcap_{n=1}^\infty Y_n$. We have, $$\mathsf{d}(X / Y) \leqslant \sup_{n\in \mathbb{N}} \mathsf{d}(X/Y_n) < \mathsf{d}(X)$$
because $\mathsf{d}(X)$ has uncountable cofinality. In particular, $Y$ is non-zero, however $\|T\vert_Y\|=0$. This is, of course, impossible as $T$ is one-to-one.
\end{proof}

\begin{theorem}\label{czupakabra}
For each super-reflexive Banach space $X$ we have $\mathsf{eo}(X)=\mathsf{d}(X)$.
\end{theorem}
\begin{proof}
The separable case follows from the Elton--Odell theorem, so let us suppose that $X$ is non-separable. Since every singular cardinal number is a limit of an increasing transfinite sequence of regular cardinals, without loss of generality we may suppose that $\lambda=\mathsf{d}(X)$ itself is regular. In particular, it has uncountable cofinality.\medskip

Fix any positive number $c<\frac{1}{2}$ and let $p\in (1,\infty)$ be so that the first inequality in \eqref{superjames} holds true. Keeping the above notation, let us say that the operator $T$ is {\it bounded by a pair} $(\gamma,\delta)$ if $\n{T}\leq\delta$ and $\gamma\leq\n{T\vert_Y}$ for every subspace $Y\subseteq X$ with $\mathsf{d}(X/Y)<\lambda$.\medskip

First, we {\it claim} that our assertion will follow whenever we show that there exists a~pair $(\gamma,\delta)$ with $\gamma/\delta>2^{-1/p}$ and such that $T$ is bounded by $(\gamma,\delta)$. To see this choose a unit vector $x_0\in X$ such that $\|Tx_0\|\geqslant \gamma$. Let $\beta<\lambda$ and suppose that we have already chosen unit vectors $x_\alpha\in X$ ($\alpha<\beta$) such that $\|Tx_\alpha\|\geqslant \gamma$ and the supports of $Tx_\alpha$ ($\alpha<\beta$) are pairwise disjoint. We are now in a~position to choose a~unit vector $x_\beta\in X$ such that $\|Tx_\beta\|\geqslant \gamma$ and the support of $Tx_\beta$ is disjoint from the supports of $Tx_\alpha$ for all $\alpha<\beta$. \medskip

Indeed, if this were impossible, we could build a transfinite, linearly independent sequence of unit vectors $(w_\xi)_{\xi<\lambda}$ in $X$ with the following properties:\medskip
\begin{romanenumerate}\item\label{pair1} $\|Tw_\xi\|\geqslant \gamma$ ($\xi<\lambda$);
\item for each $\xi<\lambda$ there exists $\eta<\beta$ such that ${\rm supp}\,Tw_\xi \cap {\rm supp}\,Tx_\eta\neq \varnothing$.
\end{romanenumerate}\medskip
The possibility of choosing a transfinite sequence $(w_\xi)_{\xi<\lambda}$ that satisfies \eqref{pair1} follows from the fact that $T$ is bounded by $(\gamma,\delta)$.\medskip

For each $\omega\leqslant \alpha<\lambda$ the space $(P_{\alpha+1}-P_\alpha)(X)$, as a subspace of $X$, is reflexive and has density at most equal to the cardinality of ${\alpha}$, hence so does its dual space. Therefore, as $\lambda>|\beta|$, there must exist:\medskip
\begin{romanenumerate}
\item $\alpha_0\in\bigcup_{\xi<\beta}\mathrm{supp}\,Tx_\xi$,\smallskip
\item a norm-one functional $x^\ast\in (P_{\alpha_0+1}-P_{\alpha_0})(X)^\ast$,\smallskip
\item a~subsequence $(w_{\xi_n})_{n=1}^\infty$ of $(w_\xi)_{\xi<\lambda}$, and\smallskip
\item a positive constant $d$
\end{romanenumerate}\medskip
such that $|\langle x^\ast, z_{\xi_n}\rangle|\geqslant d$ for all $n\in \mathbb{N}$, where $z_{\xi_n} = (P_{\alpha_0+1} - P_{\alpha_0})(w_{\xi_n})$ ($n\in \mathbb{N}$). Passing to a further subsequence and replacing $x^*$ with $-x^*$ if necessary, we may suppose that for all $n\in \mathbb{N}$ we have $\langle x^*, z_{\xi_n}\rangle \geqslant d$.\medskip

Note that $$\|Tw_{\xi_1}+\ldots + Tw_{\xi_n}\|\geqslant \langle x^*, z_{\xi_1}+\ldots + z_{\xi_n}\rangle \geqslant nd\qquad (n\in \mathbb{N}).$$
By James' inequality \eqref{superjames},
$$\n{w_{\xi_1}+\ldots + w_{\xi_n}} \leqslant C\cdot n^{1/q}\qquad (n\in \mathbb{N})$$
for some $C>0$ and $q\in (1,\infty)$ independent of $n$. Setting $$y_n= \frac{w_{\xi_1}+\ldots + w_{\xi_n}}{\n{w_{\xi_1}+\ldots + w_{\xi_n}}}\qquad(n\in\mathbb{N}),$$ we conclude that $\|Ty_n\|\geqslant \frac{dn}{Cn^{1/q}}\to \infty$ as $n\to\infty$; a~contradiction.\medskip

Now, for distinct $x_\alpha, x_\beta$ ($\alpha,\beta<\lambda$) we have
$$\|x_\alpha - x_\beta\| \geqslant\frac{1}{\delta} \|Tx_\alpha - Tx_\beta\| = \frac{1}{\delta}\bigl(\|Tx_\alpha\|^p + \|Tx_\beta\|^p\bigr)^{1/p} \geqslant \frac{\gamma}{\delta}\cdot 2^{1/p}>1.$$
This shows that our assertion is true whenever $T$ is bounded by a pair $(\gamma, \delta)$ satisfying $\gamma/\delta>2^{-1/p}$.\medskip

Set $\delta_0=\frac{1}{c}$. Since $\n{T}\leq \delta_0$, $T$ is bounded by a pair $(\gamma_0,\delta_0)$, where $$\gamma_0=\sup\bigl\{\gamma\colon \n{T\vert_Y}\geqslant\gamma\mbox{ for every subspace }Y\subseteq X\mbox{ with }\mathsf{d}(X/Y)<\lambda\bigr\}.$$
If $\gamma_0>2^{-1/p}\delta_0$, then we are done by the first part of the proof. Assume the opposite, that is, the supremum at the right-hand side is at most $2^{-1/p}\delta_0$. Fix any sequence $(\eta_n)_{n=0}^\infty$ of real numbers strictly larger than $1$ and such that $\prod_{n=0}^\infty\eta_n$ converges. Since 
$$
\gamma_0<2^{-1/p}\eta_0\delta_0=:\delta_1,
$$
there exists a subspace $X_1\subset X$ with $\mathsf{d}(X/X_1)<\lambda$ and such that $\n{T\vert_{X_1}}\leq\delta_1$. Now, define $\gamma_1$ analogously as $\gamma_0$ replacing $X$ by $X_1$. Applying again the first part of the proof to $T\vert_{X_1}$ we know that the proof is accomplished whenever $\gamma_1>2^{-1/p}\delta_1$. If this is not the case, we continue our process. At the $n^{{\rm th}}$ step we have $\delta_n=2^{-1/p}\eta_{n-1}\delta_{n-1}$, whence
$$
\delta_n=2^{-n/p}\delta_0\cdot\prod_{j=0}^{n-1}\eta_j\xrightarrow[n\to\infty]{}0.
$$
On the other hand, Lemma \ref{gamma_min} gives $\gamma_n\geq\gamma_{\mathrm{min}}$ for each $n\in\N$. Hence, if our process did not terminate, we would have arrived at the absurdity that for each $n\in \mathbb{N}$ $$\gamma_{\min}\leqslant \gamma_n \leqslant 2^{-\tfrac{1}{p}}\cdot \delta_n \leqslant \delta_n,$$ which completes the proof.
\end{proof}

\begin{remark}One cannot extend the above technique to the class of all reflexive spaces. Indeed, H\'ajek (\cite{hajek}) constructed a non-separable Tsirelson-like reflexive space $X$ whose no non-separable subspace admits an injective, bounded linear map into $\ell_p(\lambda)$ for some uncountable cardinal number $\lambda$. Nonetheless, it is easily verifiable that $\mathsf d(X) = \mathsf{eo}(X)=\omega_1$.\end{remark}

\begin{remark}The supremum appearing in the definition of $\mathsf{eo}(X)$ need not be attained, even in the case where $X$ is reflexive. Indeed, let $(p_n)_{n=1}^\infty$ be a sequence of real numbers with $p_1>1$ that increase to $\infty$ as $n\to \infty$. Consider the Banach space
$$X = \Big( \bigoplus_{n\in \mathbb{N}} \ell_{p_n}(\omega_n)\Big)_{\ell_2}.$$
Then $X$ is reflexive, as an $\ell_2$-sum of reflexive spaces, and $\mathsf{d}(X)=\omega_\omega$. For each $p\in [1,\infty)$, infinite subsets of the unit sphere of $\ell_p$ are separated by at most $2^{1/p}$ (\cite[Theorem 16.9]{wells}) and since $2^{1/p}\to 1$ as $p\to \infty$, we conclude that the unit sphere of $X$ does not contain $(1+\varepsilon)$-separated subsets of cardinality $\omega_\omega$ ($\varepsilon>0$).\medskip 

The space $X$ is not super-reflexive, though. It would interesting to find out what happens in the super-reflexive case. \end{remark}
\subsection{Banach spaces which are duals of WLD spaces of large density}
The last result of this chapter is devoted to Banach spaces that are duals of WLD spaces of density bigger than $\mathfrak{c}$. Here, by a~{\it weakly Lindel\"of determined} Banach space (\emph{WLD}, for short) we understand a Banach space $X$ such that for some set $\Gamma$ the dual ball $(B_{X^\ast},w^\ast)$ is homeomorphic to a~subset of $$\{x\in \mathbb{R}^\Gamma\colon \text{the set }\{\gamma\in \Gamma\colon x(\gamma)\neq 0\}\text{ is countable}\},$$ that is compact in the topology of point-wise convergence. This is a~rich class of Banach spaces; for instance, every weakly compactly generated space is WLD (for more information consult the monograph \cite{hajek_book}). \medskip

Following \cite{kubis}, by a~$1$-{\it projectional skeleton} in $X$ we  understand any family $\mathfrak{s}=\{R_s\}_{s\in\Gamma}$ of norm-one projections on $X$, indexed by a~directed partially ordered set $\Gamma$ such that:
\begin{itemize*}
\item[(a)] $X=\bigcup_{s\in\Gamma}R_s(X)$, where each subspace $R_s(X)$ ($s\in \Gamma$) is separable;
\item[(b)] if $s\leq t$ ($s,t\in \Gamma$), then $R_s=R_sR_t=R_tR_s$;
\item[(c)] if $s_1\leq s_2\leq\ldots$ is a sequence in $\Gamma$, then $t=\sup_{n\in \mathbb{N}}s_n$ exists in $\Gamma$ and $$R_t(X)=\oo{\bigcup_{n=1}^\infty R_{s_n}(X)}.$$
\end{itemize*}
Then, for every directed subset $T\subseteq\Gamma$ the pointwise limit $R_Tx:=\lim_{s\in T}R_sx$ ($x\in X$) exists and defines a~projection $R_T$ onto the closure of $\bigcup_{s\in T}R_s(X)$ ({\it cf. }\cite[Lemma~11]{kubis}).\medskip

Set $D(\mathfrak{s})=\bigcup_{s\in\Gamma}R_s^\ast(X^\ast)$. Then $D(\mathfrak{s})$ is a~$1$-norming subspace of $X^\ast$. According to \cite[Theorem~15]{kubis}, whenever we have a~$1$-norming set $D\subseteq X^\ast$ which generates projections, there exists a~$1$-projectional skeleton $\mathfrak{s}$ in $X$ with $D\subseteq D(\mathfrak{s})$. Secondly, for every WLD space $X$ its dual $X^\ast$ generates projections, therefore any such space $X$ admits a~$1$-projectional skeleton $\mathfrak{s}$ for which $D(\mathfrak{s})=X^\ast$ (for more details, see \cite{kubis} or \cite[Ch.~17]{kakol}). Our next proof shows, in a~sense, how far one can go with the original argument of Kottman. We shall also employ a~technique, due to Kubi\'s ({\it cf. }~\cite[Theorem~12]{kubis}), of building PRI's out of projectional skeletons.

\begin{theorem}\label{long_PRI}
Let $X$ be a WLD Banach space with $\mathsf{d}(X)>\mathfrak{c}$. Then $\mathsf{k}(X^\ast)$ is uncountable. 
\end{theorem}

\begin{proof}
%If $(P_\alpha)_{\omega\leq\alpha<\lambda}$ is a~PRI for $X$, then \cite[Lemma~3]{john_zizler} yields that $(P_\alpha^\ast)_{\omega\leq\alpha<\lambda}$ forms a~PRI for $X^\ast$ and, of course, all the projections $(P_\alpha^\ast)_{\omega\leqslant \alpha<\lambda}$ are $w^\ast$-to-$w^\ast$ continuous.\medskip
Let $\lambda=\mathsf{d}(X)$; we may suppose that $\lambda=\mathfrak{c}^+$, the successor of the continuum. Pick any $1$-projectional skeleton $\mathfrak{s}=\{R_s\}_{s\in\Gamma}$ in $X$ such that $D(\mathfrak{s})=X^\ast$. We will construct inductively:
\begin{itemize}
\item $\{T_\alpha\}_{\omega\leq\alpha\leq\lambda}$, a~continuous chain of upwards-directed subsets $T_\alpha\subset\Gamma$ such that
\begin{itemize}\item $T_\alpha\subset T_{\alpha+1}$ ($\omega\leqslant \alpha \leqslant \lambda$),
\item $\abs{T_\alpha}\leq\max\{|\alpha|, \omega\}$ ($\omega\leqslant \alpha \leqslant \lambda$), and 
\item $\bigcup_{s\in T_\lambda}R_s(X)$ is dense in $X$\end{itemize} (here, by a {\it continuous chain}, we understand a chain such that for every limit ordinal $\delta\leq\lambda$ we have $T_\delta=\bigcup_{\alpha<\delta}T_\alpha$; the density condition is not really essential, however, in this way we may guarantee that the resulting sequence of corresponding projections gives a~PRI in $X$),
\item $(\xi_\alpha)_{\omega\leq\alpha<\lambda}$, a~strictly increasing sequence of ordinals less than $\lambda$ such that $|T_\alpha|\leqslant |\xi_\alpha|$ ($\omega\leqslant \alpha \leqslant \lambda$),
\item $(x_\alpha)_{\omega\leq\alpha<\lambda}$, a sequence of norm-one vectors in $X$,
\item $(x_\alpha^\ast)_{\omega\leq\alpha<\lambda}$, a~sequence of norm-one functionals in $X^\ast$
\end{itemize}
so that, under the notation
\begin{equation}\label{XP}
X_\alpha=\oo{\bigcup_{s\in T_\alpha}R_s(X)}\,\,\,\,\mbox{ and }\,\,\, P_\alpha=\lim_{s\in T_\alpha}R_s\quad (\omega<\alpha\leq\lambda),
\end{equation}
the following conditions are satisfied:
\begin{itemize}
\item[(i)] $x_\alpha\in (P_{\xi_\alpha+1}-P_{\xi_\alpha})(X)$;

\item[(ii)] $x_\alpha^\ast\in P_{\xi_\alpha+2}^\ast(X^\ast)$;

\item[(iii)] $\langle x_\alpha,x_\alpha^\ast\rangle=1$;

\item[(iv)] $x_\alpha\in\bigcap_{\beta<\alpha}\mathrm{ker}(x_\beta^\ast)$, for each $\omega\leq\alpha<\lambda$.
\end{itemize}

First, choose any dense subset $\{d_\alpha\}_{\omega\leq\alpha<\lambda}$ of $X$. We set $T_\omega=\varnothing$, $X_\omega=\{0\}$ and $P_\omega=0$. Now, pick any countable upwards-directed set $T_{\omega+1}$ with $d_\omega\in X_{\omega+1}$ ($X_\alpha$'s and $P_\alpha$'s will always be given by \eqref{XP}). Set $\xi_\omega=\omega$ and pick $x_\omega$ as any unit vector from $X_{\omega+1}$. Choose a~norm-one functional $x_\omega^\ast$ with $\langle x_\omega,x_\omega^\ast\rangle=1$; since $D(\mathfrak{s})=X^\ast$ we may find an $s\in\Gamma$ with $x_\omega^\ast\in R_s^\ast(X^\ast)$. Thus, $x_\omega^\ast$ may be written in the form $z^\ast\circ R_s$ for some $z^\ast\in X^\ast$. We now take $T_{\omega+2}$ to be a~countable, upwards-directed subset of $\Gamma$ such that $\{s\}\cup T_{\omega+1}\subseteq T_{\omega+2}$ and $d_{\omega+1}\in X_{\omega+2}$. It is easy to see that $X_{\omega+2}$ is just the range of $P_{\omega+2}$, thus $x_\omega^\ast\in P_{\omega+2}^\ast(X^\ast)$. Hence, all the conditions (i)--(iv) are established for $\alpha=\omega$ (the last one being vacuous at the moment).

Given any $\delta\in (\omega,\lambda)$, suppose we have already defined: 
\begin{itemize}
\item ordinals $\xi_\alpha$ for $\alpha<\delta$, 
\item sets $T_\beta$ as above, for all indices $\beta$ satisfying $\beta\leq\xi_\alpha+2$ for some $\alpha<\delta$ (and hence also all corresponding $X_\alpha$'s and $P_\alpha$'s),
\item vectors $x_\alpha$ and functionals $x_\alpha^\ast$ for $\alpha<\delta$,
\end{itemize}
so that conditions (i)--(iv) hold true for every $\alpha<\delta$. 

\medskip
If $\delta$ is a~limit ordinal, we need to define $T_{\zeta}$ for $\zeta:=\sup_{\alpha<\delta}\xi_\alpha$. Of course, we set $T_\zeta=\bigcup_{\beta<\zeta}T_\beta$. Notice that $P_\zeta^\ast(X^\ast)$ is then a~weak$^\ast$ closed subspace of $X^\ast$ which does not separate points of $X$ (in fact, we have $^\perp P_\zeta^\ast(X^\ast)=\mathrm{ker}P_\zeta$), so that we can find a~unit vector $x_\delta\in\bigcap_{\beta<\delta}\mathrm{ker}(x_\beta^\ast)$. We set $\xi_\delta=\zeta$ and pick an upwards-directed set $T_{\xi_\delta+1}\subset\Gamma$ such that $T_{\xi_\delta}\subset T_{\xi_\delta+1}$, $\abs{T_{\xi_\delta+1}}\leq |\xi_\delta|$ and $\{d_{\xi_\delta},x_\delta\}\subset X_{\xi_\delta+1}$. Finally, we find a~norm-one functional $x_\delta^\ast$ with $\langle x_\delta,x_\delta^\ast\rangle=1$, pick an index $s\in\Gamma$ so that $x_\delta^\ast\in R_s^\ast(X^\ast)$ and next we choose an upwards-directed set $T_{\xi_\delta+2}\subset\Gamma$ satisfying: $\{s\}\cup T_{\xi_\delta+1}\subset T_{\xi_\delta+2}$, $\abs{T_{\xi_\delta+2}}\leq|\xi_\delta|$ and $d_{\xi_\delta+1}\in X_{\xi_\delta+2}$.

\medskip
If $\delta$ is a successor, say $\delta=\gamma+1$, we pick a~unit vector $x_\delta\in\mathrm{ker}P_\gamma$ and then we just repeat the scheme from the first step in an obvious manner.\medskip
Having guaranteed all the conditions (i)--(iv), notice that we can build an Auerbach system in $X\times X^\ast$ of length $\lambda$ in the following way:
For each $\alpha\in [\omega,\lambda)$ we define a~functional $\w{x_\alpha}\in X^{\ast\ast}$ as the composition of the evaluation functional at $x_\alpha$ with the projection $P_{\xi_\alpha+2}^\ast$, that is,
$$
\langle x^\ast,\w{x_\alpha}\rangle=\langle x_\alpha, P_{\xi_\alpha+2}^\ast x^\ast\rangle\qquad (x^\ast\in X^\ast).
$$
By virtue of (i)--(iv), we have:
$$
\langle x_\beta^\ast,\w{x_\alpha}\rangle=\langle x_\alpha,P_{\xi_\alpha+2}^\ast x_\beta^\ast\rangle=\left\{\begin{array}{llll}
\langle x_\alpha,0\rangle&=&0 & \mbox{if }\alpha<\beta,\\
\langle x_\alpha,x_\alpha^\ast\rangle&=& 1 & \mbox{if }\alpha=\beta,\\
\langle x_\alpha,x_\beta^\ast\rangle&=&0 & \mbox{if }\alpha>\beta.
\end{array}\right.
$$
Therefore, we have obtained a system $\{(x_\alpha^\ast, \w{x_\alpha})\}_{\omega\leq\alpha<\lambda}\subset X^\ast\times X^{\ast\ast}$ satisfying the following conditions:
\begin{itemize}
\item[(v)] $\n{x_\alpha^\ast}=\n{\w{x_\alpha}}=1$ for each $\omega\leq\alpha<\lambda$;
\item[(vi)] $\langle x_\beta^\ast, \w{x_\alpha}\rangle=\delta_{\alpha\beta}$ for all $\omega\leq\alpha,\beta<\lambda$;
\item[(vii)] $\w{x_\alpha}$ is weak$^\ast$-continuous for each $\omega\leq\alpha<\lambda$,
\end{itemize}
that is, $\{(\w{x_\alpha},x_\alpha^\ast)\}_{\omega\leq\alpha<\lambda}$ forms an Auerbach system in $X\times X^\ast$ of length $\lambda$.

\medskip
Now, define 
$$
\mathcal{A}=\Bigl\{\bb{x}=(\langle x^\ast,\w{x_\alpha}\rangle)_{\omega\leq\alpha<\lambda}\in\R^{[\omega,\lambda)}\colon\, \n{x^\ast}\leq 1\mbox{ and }\langle x^\ast,\w{x_\alpha}\rangle\not=0\mbox{ for countably many }\alpha\mbox{'s}\Bigr\}.
$$
Plainly, $\mathcal{A}$ enjoys the following three properties:
\begin{itemize*}
\item[(a1)] $\bb{e}_\alpha\in \mathcal{A}$ for every $\omega\leq\alpha<\lambda$, where $\bb{e}_\alpha$ is the $\alpha^{{\rm th}}$ vector from the canonical basis of the vector space $\R^{[\omega,\lambda)}$;
\item[(a2)] if $\bb{x}\in \mathcal{A}$, then $-\bb{x}\in \mathcal{A}$;
\item[(a3)] $\supp(\bb{x}):=\{\omega\leq\alpha<\lambda\colon \bb{x}(\alpha)\not=0\}$ is countable for every $\bb{x}\in \mathcal{A}$. (We write $\bb{x}(\alpha)$ for $\langle x^\ast,\w{x_\alpha}\rangle$, the $\alpha^{{\rm th}}$ coordinate of $\bb{x}$.)
\end{itemize*}
Assume, in search of a contradiction, that every uncountable subset of the unit ball of $X^\ast$ contains two distinct elements at distance not larger than $1$. Then, the set $\mathcal{A}$ satisfies also the additional condition:
\begin{itemize*}
\item[(a4)] for every uncountable set $\mathcal{B}\subseteq \mathcal{A}$ there exist $\bb{x},\bb{y}\in \mathcal{B}$ with $\bb{x}\not=\bb{y}$ such that $\bb{x}-\bb{y}\in \mathcal{A}$.
\end{itemize*}

For any two sequences $\bb{x},\bb{y}\in \mathcal{A}$ we shall say that $\bb{y}$ {\it extends} $\bb{x}$ if and only if there exists a~third sequence $\bb{z}\in \mathcal{A}$ such that the following conditions are satisfied:
\begin{itemize*}
\item[(e1)] $\bb{x}(\alpha)=\bb{y}(\alpha)=\bb{z}(\alpha)$ for each $\alpha\leqslant \sup \supp(\bb{x})$;
%\item[(e2)] $\bb{y}(\alpha)=\bb{z}(\alpha)$ for each $\alpha\in\supp(\bb{x})$;
\item[(e2)] there is an ordinal number $\beta$ with $\sup(\supp(\bb{x}))<\beta<\lambda$ such that $\bb{y}(\beta)>0$ and $\bb{z}(\beta)=-1$.
\end{itemize*}
In such a case we say that $\bb{z}$ is a~{\it witness} of $\bb{y}$ extending $\bb{x}$ at the $\beta^{{\rm th}}$ coordinate. By a~{\it chain} starting with an element $\bb{x}\in \mathcal{A}$ we mean any sequence $(\bb{x}_\alpha)_{0\leq\alpha<\xi}$, where $\xi\geq 0$ is an ordinal number, for which $\bb{x}_0=\bb{x}$ and $\bb{x}_{\alpha+1}$ extends $\bb{x}_\alpha$ for every $\alpha\geq 0$ with $\alpha+1<\xi$.

\medskip
Let $(\bb{x}_\alpha)_{0\leq\alpha<\xi}$ be a chain and, for any $\alpha\geq 0$ with $\alpha+1<\xi$, let $\bb{z}_\alpha$ be a~witness of $\bb{x}_{\alpha+1}$ extending $\bb{x}_\alpha$ at the $\beta_\alpha^{{\rm th}}$ coordinate. Then, according to (e2), we have $\bb{z}_\alpha(\beta_\alpha)=-1$ and $\bb{x}_{\alpha+1}(\beta_\alpha)>0$ for all $\alpha$'s as above. Moreover, by (e1), we obtain $\bb{z}_\gamma(\beta_\alpha)>0$ whenever $\gamma>\alpha$. Consequently, for any two ordinals $\alpha$, $\beta$ with $1\leq\alpha+1<\beta+1<\xi$ at least one of the coordinates of $\bb{z}_\beta-\bb{z}_\alpha$ is larger than one, hence $\bb{z}_\beta-\bb{z}_\alpha\not\in \mathcal{A}$. Therefore, condition (a4) implies that every chain in $\mathcal{A}$ is at most countable. On the other hand, in view of the Kuratowski--Zorn lemma, for every $\bb{x}\in \mathcal{A}$ there exists a~maximal chain starting with $\bb{x}$. Now, we {\it claim} that every such maximal chain contains a~maximal element, \emph{i.e.}, a~sequence $\bb{y}\in A$ which  extends $\bb{x}$ and which has no further extension.\medskip

Indeed, let $\mathscr{C}$ be a~maximal chain starting with some $\bb{x}\in \mathcal{A}$ for which there is no maximal extension. Then, there is a~countable sequence $(\bb{y}_n)_{n=1}^\infty\subset\mathscr{C}$ so that $\bb{y}_{n+1}$ extends $\bb{y}_n$ for each $n\in\N$, but there is no $\bb{y}\in\mathscr{C}$ extending all $\bb{y}_n$'s. Let $y_n^\ast\in B_{X^\ast}$ be given so that
$$
\bb{y}_n=(\langle y_n^\ast,\w{x_\alpha}\rangle)_{\omega\leq\alpha<\lambda}\qquad (n\in\N).
$$
Take $z^\ast$ to be any weak$^\ast$-cluster point of $\{y_n^\ast\colon n\in\N\}$. Then, obviously $z^\ast$ lies in the unit ball of $X^\ast$ and $\langle z^\ast,\w{x_\alpha}\rangle=0$ for all but countably many $\alpha$'s, since the same property is shared by each $y_n^\ast$. Thus, $z^\ast$ gives rise to the element $\bb{z}=(\langle z^\ast,\w{x_\alpha}\rangle)_{\omega\leq\alpha<\lambda}$ of $\mathcal{A}$ which extends each $\bb{y}_n$ by the very definition. (Here, we employed the weak$^\ast$-continuity of $P_\alpha^\ast$'s.)

\medskip
Now, we define by transfinite recursion a~sequence $(\bb{x}_\alpha)_{\omega\leq\alpha<\lambda}$ of maximal extensions as follows. First, let $\bb{x}_0$ be any maximal extension of $\bb{e}_\omega$. Now, if $\beta<\lambda$ and all the terms $\bb{x}_\alpha$, for $\omega\leq\alpha<\beta$, have been already defined, we pick any ordinal $\gamma$ with $$\sup\, \bigcup_{0\leq\alpha<\beta}\supp(\bb{x}_\alpha)<\gamma<\lambda$$and take $\bb{x}_\beta$ to be any maximal extension of $\bb{e}_\gamma$. (Recall that all maximal chains are countable so the entire support of all $\bb{x}_\alpha$'s is not cofinal in $\lambda$ because $\lambda$, being a successor cardinal, is regular; see \cite[Corollary~5.3]{jech}; we have actually more as the vectors $(\bb{x}_\alpha)_{\omega\leq\alpha<\lambda}$ are disjointly supported.) The rest is the same as in Kottman's proof; just instead of Ramsey's theorem we need its variation for larger cardinals, that is, the Erd\H{o}s--Rado theorem: $$\mathfrak{c}^+\to (\omega_1)_2^2$$ ({\it cf.} \cite[Theorem~9.6]{jech}). So, define a colouring $\mathsf{c}\colon [[\omega,\lambda)]^2\to\{0,1\}$ of all $2$-element subsets of $[\omega,\lambda)$ as follows
$$\mathsf{c}(\{\alpha,\beta\})=\left\{\begin{array}{ll}
0 & \mbox{if }\bb{x}_\alpha-\bb{x}_\beta\in \mathcal{A},\\
1 & \mbox{if }\bb{x}_\alpha-\bb{x}_\beta\not\in \mathcal{A}.
\end{array}\right.
$$
Note that by property (a2) this definition is well-posed. By the Erd\H{o}s--Rado theorem, there exists an uncountable set $\mathcal{B}\subset \mathcal{A}$ such that $$\text{either}\quad[\mathcal{B}]^2\subset \mathsf{c}^{-1}(\{0\})\quad\text{or}\quad[\mathcal{B}]^2\subset \mathsf{c}^{-1}(\{1\}).$$ However, according to (a4), only the former possibility may occur, and by relabeling we may assume that $\mathcal{B}=\mathcal{A}$. Property (a4) implies that there are two ordinal numbers $\alpha$ and $\beta$ with $\omega\leqslant \alpha<\beta< \lambda$ and such that $(\bb{x}_\alpha-\bb{x}_{\alpha+1})-(\bb{x}_\beta-\bb{x}_{\beta+1})\in A$. But also $\bb{x}_\alpha-\bb{x}_{\beta+1}\in A$ which shows that $\bb{x}_\alpha-\bb{x}_{\beta+1}$ is a~witness of  $\bb{x}_\alpha-\bb{x}_{\alpha+1}-\bb{x}_\beta+\bb{x}_{\beta+1}$ extending $\bb{x}_\alpha$; a~contradiction with the maximality of $\bb{x}_\alpha$.
\end{proof}

\begin{remark}
As it was shown by Fabian and Godefroy (\cite{fabian_godefroy}), the dual of any Asplund space always admits a~PRI. However, it may happen that the corresponding projections $P_\alpha$'s fail to be continuous in the weak$^\ast$ topology (for a~specific example, see \cite[Remark~5]{fabian_godefroy}). This shows that our hypothesis, implying the existence of a~special PRI in $X$ itself, was crucial for our considerations.
\end{remark}
%%%%%%%%%%%%%%%%%%%%%%%%%%%%%%%%%%%%%%%%%%%%%%%%%%%%%%%

\section{Spaces of continuous functions}\label{finalsection}

\subsection{Compact Hausdorff spaces} Throughout this section $K$ stands for an infinite, compact, Hausdorff space unless otherwise stated. By $C(K)$ we denote the familiar Banach space of scalar-valued continuous functions on $K$ furnished with the supremum norm. \medskip

With the aid of the Urysohn lemma, it is easy to build a 2-separated (hence, by the triangle inequality, equilateral) sequence in the unit sphere of $C(K)$. Thus it is natural to ask whether $\mathsf{eo}(C(K))>\omega$ in the case where $K$ is non-metrisable. Recently, Koszmider (\cite{kosz}) established independence from the axioms of set theory of the answer to the above question. We offer therefore a~number of sufficient conditions implying uncountability of $\mathsf{eo}(C(K))$.

\begin{proposition}\label{subspaces}Suppose that $\mathsf{eo}(C(K))$ is countable. Then, given a closed subset $L\subseteq K$, $\mathsf{eo}(C(L))$ is countable too. Moreover, if $C(L)$ contains a 2-equilateral set of cardinality $\kappa$, then so does $C(K)$.\end{proposition}
\begin{proof}Arguing by contraposition, let $L\subseteq K$ be a closed subspace and suppose that there exists an uncountable $(1+\varepsilon)$-separated family $\{f_i\colon i\in I\}$ of unit vectors in the sphere of $C(L)$. Let $\hat{f}_i$ be a~norm-preserving Tietze--Urysohn extension of $f_i$ to $K$ ($i\in I$). Then $\{\hat{f}_i\colon i\in I\}$ witnesses uncountability of $\mathsf{eo}(C(K))$. \end{proof}
A similar proof gives the following result.
\begin{proposition}\label{subspaces2}Suppose that $\mathsf{k}(C(K))$ is countable. Then, given a closed subset $L\subseteq K$, $\mathsf{k}(C(L))$ is countable too. \end{proposition}

Certainly (hereditary) separability of $K$ is a necessary condition for countability of $\mathsf{eo}(C(K))$. To see this, note that if $K$ is non-separable, there exist families $$\{x_\alpha\colon \alpha<\omega_1\}\quad \text{and}\quad \{U_\alpha\colon \alpha<\omega_1\}$$ consisting of distinct points in $K$ and open subsets of $K$, respectively, with the property that for each $\alpha<\omega_1$ we have $x_\alpha\in U_\alpha$ and $x_\beta\notin U_\alpha$ for all $\beta<\alpha$. Indeed, the closure of each countable set in $K$ is a proper subset of $K$ and by complete regularity proper closed subsets can be separated by open sets from elements in the complement. Using Urysohn's lemma we may thus build continuous functions $f_\alpha\colon K\to [-1,1]$ such that $f(x_\alpha) = 1$ and $f(x) = -1$ for all $x\in K\setminus U_\alpha$. It then follows that $\{f_\alpha\colon \alpha<\omega_1\}$ is a 2-separated subset of the unit sphere of $C(K)$. Taking into account  Proposition~\ref{subspaces} and the Tietze--Urysohn Extension Theorem, we arrive at the following conclusion.
\begin{proposition}\label{hs}Suppose that $K$ is not hereditarily separable. Then $\mathsf{eo}(C(K))$ is uncountable. \end{proposition}

The case where $K$ is totally disconnected is even easier. Indeed, take an uncountable collection of clopen subets sets $U_i$ ($i\in I$), in which case
$$\{\mathds{1}_{U_i} - \mathds{1}_{K\setminus U_i}\colon i \in I\}$$
forms an uncountable 2-separated (hence equilateral) subset of the unit sphere in $C(K)$. It is a standard fact from the point-set topology that $\mathsf d(C(K))$ is equal to $w(K)$, the minimal cardinality of a basis for $K$ (\emph{the weight}) if $K$ is infinite and $\omega$, otherwise. Let us then record formally the following observation.

\begin{proposition}\label{td}Suppose that $K$ contains a closed, totally disconnected subspace $L$. Then $\mathsf{eo}(C(K))\geqslant \mathsf{d}(C(L))$. \end{proposition}
One may wonder whether each non-metrisable compact, Hausdorff space $K$ contains a~non-metrisable, totally disconnected closed subspace---this in the light of Proposition~\ref{td}, would be sufficient to conclude uncountability of $\mathsf{eo}(C(K))$. Nyikos (\cite[Example 6.17]{nyikos}) constructed, assuming Jensen's Diamond Principle $\diamondsuit$, a non-metrisable, compact manifold $K$ whose each non-metrisable, closed subspace contains a copy of the unit interval (thus, it is not totally disconnected). Therefore, one cannot hope to prove such a~theorem about totally disconnected subspaces in $\mathsf{ZFC}$ only (this also follows from the main result of \cite{kosz}). However Nyikos' manifold is non-separable, hence by Proposition~\ref{hs}, the number $\mathsf{eo}(C(K))$ is uncountable. \medskip

We conjecture that, assuming the Proper Forcing Axiom, every non-metrisable, perfectly normal compact space $K$ contains a non-metrisable totally disconnected closed subspace\footnote{Added in proof: Very recently, Koszmider has constructed a~\textsf{ZFC} example of a non-metrisable, compact Hausdorff space whose each totally disconnected closed subspace is metrisable (\texttt{arXiv:1509.05282}), thereby answering our question in the negative.}---according to Proposition~\ref{td}, this would be sufficient to derive uncountability of $\mathsf{eo}(C(K))$.\medskip

The existence of a non-separable Radon measure on $K$ is also sufficient for uncountability of $\mathsf{eo}(C(K))$.

\begin{proposition}\label{maharam}Suppose that there exists a positive measure $\mu\in C(K)^*$ such that $L_1(\mu)$ is non-separable. Then $$\mathsf{eo}(C(K))\geqslant \mathsf d(L_1(\mu))>\omega.$$ 
\end{proposition}

\begin{proof}For a Radon measure $\mu$ on a compact space, let $\lambda = \mathsf d(L_1(\mu))$. The Hilbert space $\ell_2(\lambda)$ embeds isometrically into $L_1(\mu)$. Indeed, by Maharam's theorem (\cite[Theorem 331L]{fremlin}), $L_1(\mu)$ is isometric to $L_1(\{0,1\}^\lambda)$. (Here, $L_1(\{0,1\}^\lambda)$ denotes the $L_1$-space with respect to the Haar measure on the Cantor group $\{0,1\}^\lambda$.) Consequently, we can find a~collection of $\lambda$ many independent Gaussian random variables with mean zero and the same variance; their linear span is isometric to $\ell_2(\lambda)$. \medskip

Thus, by simple duality, $\ell_2(\lambda)$ is a quotient of $C(K)$. Since $\mathsf{eo}(\ell_2(\lambda))=\lambda$ (any orthonormal basis is a witness of this fact), by Proposition~\ref{lifting}, the conclusion follows.\end{proof}

\begin{remark}The conclusion of Proposition~\ref{maharam} extends to the class of all non-reflexive Grothendieck spaces. (A Banach space $X$ is \emph{Grothendieck} if each bounded, linear operator $T\colon X\to c_0$ is weakly compact.) Indeed, Haydon (\cite[Theorem~1]{haydon}) proved that the dual space of every non-reflexive Grothendieck space contains an~isometric copy of $L_1(\{0,1\}^{\omega_1})$, so in this case we argue exactly as in the proof of Proposition~\ref{maharam} in order to produce an~uncountable $(1+\varepsilon)$-separated subset of the unit sphere of $X$; here $\varepsilon$ can be taken to be arbitraily close to $\sqrt{2}-1$. \end{remark}

If $K$ is a Rosenthal compact space (that is, $K$ is homeomorphic to a compact subset of the space of first-class Baire functions on a Polish space endowed with the topology of point-wise convergence), then for each Radon measure $\mu$ on $K$, the space $L_1(\mu)$ is separable (\cite[Theorem~2]{todor}). One may then wonder whether a potential counter-example should fall into this class. This is however not the case. Indeed, by Proposition~\ref{hs}, any counter-example $K$ must be hereditarily separable, so if $K$ is also Rosenthal compact, by \cite[Theorem 4]{todor}, it must contain a copy of the split interval. An appeal to Proposition~\ref{td} yields the following conclusion.

\begin{proposition}Suppose that $K$ is a non-metrisable, Rosenthal compact space. Then $\mathsf{eo}(C(K))$ is uncountable. \end{proposition}

We say that $K$ is \emph{perfectly normal}, if each closed subset of $K$ is $\EuScript{G}_\delta$. Thus, if $U$ is an open subset of a perfectly normal space $K$, then there exists a norm-one function $f\in C(K)$ such that $f(t)>0$ for $t\in U$ and $f(t)=0$ otherwise. The following fact is a~part of \cite[Theorem 2]{mv}.

\begin{proposition}\label{lemmamv}Suppose that $K$ is not perfectly normal. Then the unit sphere of $C(K)$ contains an uncountable $2$-separated subset. In particular, $\mathsf{eo}(C(K))$ is uncountable. \end{proposition}

\begin{definition}We say that a continuous function $f\colon K\to \mathbb{R}$ is \emph{locally sign-changing} if for each non-isolated point $z\in K$ with $f(z)=0$ the function $f$ takes both positive and negative values in every neighbourhood of $z$.\end{definition}
\noindent
For instance, if $K=[-1,1]$, then $f(x)=\sin x$ is a locally sign-changing function whereas $g(x)=|x|$ is not.

\begin{lemma}\label{paranormal}Suppose that $K$ is perfectly normal. Then for each pair of distinct points $x,y\in K$ there exists a norm-one, locally sign-changing function $f\in C(K)$ such that $f(t)=1$ in some neighbourhood of $x$ and $f(t)=-1$ in some neighbourhood of $y$.\end{lemma}

\begin{proof}Let $g\in C(K)$ be a norm-one function such that $g(t)=1$ in some neighbourhood of $x$ and $g(t)=-1$ in some neighbourhood of $y$. If $g^{-1}(\{0\})$ has non-empty interior, by perfect normality of $K$, let us choose a norm-one function $h$ that is strictly positive on ${\rm int}\,g^{-1}(\{0\})$ and 0 otherwise. If the interior is already empty, take $h=0$. \medskip

Let $w=g+h$; then $w^{-1}(\{0\})$ has empty interior (in particular, it contains no isolated points). Consider all points $z\in w^{-1}(\{0\})$ such that $0\leqslant w(t)< 1/4$ for $t\in U^+_z$, where $U^+_z$ is some open neighbourhood of $z$, and let $U^+$ be the union of all such sets $U^+_z$. Take a~function $w^+$ of norm at most 3/4 that is strictly positive in $U^+$ and zero otherwise. \medskip

Similarily, consider all points $z\in w^{-1}(\{0\})$ such that $-1/4< w(t)\leqslant 0$ for $t\in U^-_z$, where $U^-_z$ is some open neighbourhood of $z$, and let $U^-$ be the union of all these sets $U^-_z$. Take a~function $w^-$ of norm at most 3/4 that is strictly negative in $U^-$ and zero otherwise.\medskip

Finally, observe that $f = w+w^+ + w^-$ is a norm-one, locally sign-changing function with the desired properties.\end{proof}\medskip

We prove that for $C(K)$-spaces with $K$ perfectly normal, the number $\mathsf k(C(K))$ is as large as possible, that is, it is equal to $d(C(K))$.

\begin{theorem}\label{weight}Suppose that $K$ is perfectly normal. Then $\mathsf k(C(K)) = \mathsf{d}(C(K))$. Moreover, if $K$ does not have any isolated points, then the unit sphere of $C(K)$ contains a~{\rm (1+)}-separated subset of cardinality $\mathsf{d}(C(K))$ that consists of locally sign-changing functions. \end{theorem}

\begin{proof}Recall that $w(K) = \mathsf{d}(C(K))$. Let us first note that we may reduce the situation to the case where $K$ does not have any isolated points. Indeed, every space decomposes uniquely into a union of a perfect set and a scattered one (one of them may be empty; see, \emph{e.g.}, \cite[Theorem 3 in \S 9]{kuratowski}), so let us write $K=L\cup S$, where $L$ is perfect (hence closed) and $S$ is scattered. As $K$, being perfectly normal, is hereditarily Lindel\"{o}f, for each ordinal $\xi$ the complement $K\setminus K^{(\xi)}$ of the $\xi^{{\rm th}}$-derived subset of $K$ is countable and $L=K^{(\eta)}$ for some countable ordinal $\eta$. Consequently, $$S=K\setminus L = K \setminus \Bigl(\bigcap_{\xi\leqslant\eta}K^{(\xi)}\Bigl)=\bigcup_{\xi\leqslant\eta}K\setminus K^{(\xi)}$$ is countable as a union of countably many  countable sets. Being an open subset of a~compact space, $S$ is locally compact. As $S$ is countable, it is metrisable so $w(S)=\omega$. As $w(K)=\max\{w(L), w(S)\}$, we infer that $L$ has the same weight as $K$. In the light of Proposition~\ref{subspaces2}, we may then suppose that $K$ does not have any isolated points.\medskip

Assume, in search of contradiction, that each (1+)-separated family consisting of norm-one, locally sign-changing functions has cardinality strictly less than $w(C(K))$ and take one, $\{f_i\colon i\in I\}$ say, that is maximal (with respect to inclusion) subject to these conditions. \medskip

We \emph{claim} that the the family $\{f_i\colon i\in I\}$ separates points in $K$. Once it is proved, the Stone--Weierstrass theorem will immediately yield that the algebra $\mathscr A$ generated by $\{f_i\colon i\in I\}$ is dense in $C(K)$. However, the density of $\mathscr A$ is at most $|I|<w(K)$, which will ultimately lead to a contradiction.\medskip

Assume that the functions $\{f_i\colon i\in I\}$ do not separate points in $K$. Thus, for some pair of distinct points $x,y\in K$ and every $i\in I$ we have $f_i(x)=f_i(y)$. By Lemma~\ref{paranormal}, we may find a norm-one, locally sign-changing function $f\in C(K)$ such that $f(t)=1$ in some neighbourhood of $x$ and $f(t)=-1$ in some neighbourhood of $y$. Fix any $i\in I$. If $f_i(x)\neq 0$, then we have
$$\|f-f_i\|\geqslant \max\bigl\{ |1-f_i(x)|, |-1-f_i(x)|\bigr\} > 1.$$
Otherwise (that is, if $f_i(x)= 0$), we use the fact that $f_i$ is locally sign-changing to pick a~point $t$ in the neighbourhood of $x$ where $f$ takes value 1 such that $f_i(t)<0$. Thus
$$\|f-f_i\|\geqslant |1 -f_i(t)|>1.$$
This is a contradiction with the maximality of $\{f_i\colon i\in I\}$. \end{proof}

It is a tantalising problem whether the hypothesis of perfect normality in Theorem~\ref{weight} may be dropped.

\subsection{Locally compact Hausdorff spaces}
We will employ Theorem~B to obtain another class of spaces for which the cardinal function $\mathsf{k}$ assumes uncountable values.

\begin{proposition}Let $K$ be a non-separable, locally compact, Hausdorff space. Then $\mathsf{k}(C_0(K))$ is uncountable. \end{proposition}
\begin{proof}Consider first the case where $K$ contains a non-metrisable, compact subset $L$. Then by Theorem~B, $\mathsf{k}(C(L))>\omega$. Using the Tietze--Urysohn extension theorem for locally compact spaces, we may extend each function $f\in C(L)$ to a function from $C_0(K)$ preserving the norm. Thus, $\mathsf k(C_0(K))$ is uncountable. \medskip

Having eliminated the case where there exists a compact, non-metrisable subset of $K$, we arrive at the case where each compact subset of $K$ is second-countable, hence separable. Assume that each (1+)-separated subset of the unit sphere of $C_0(K)$ is countable. Similarly as in the Introduction, take a set $\{f_n\colon n\in \mathbb{N}\}$ that is maximal among all (1+)-separated subsets of the unit sphere of $C_0(K)$ consisting of functions assuming the value $1$. Define
$$F=\{x\in K\colon f_n(x)\not=0\mbox{ for some }n\in\N\};$$ this is clearly a $\sigma$-compact subset of $K$. As $K$ is non-separable and its compact subsets are second-countable, $F$ is separable and hence, $\overline{F}$ is a proper subset of $K$. Consider the function $$g(x) = \sum_{k=1}^\infty \frac{|f_k(x)|}{2^k}\quad (x\in K).$$
We have $g(x)=0$ for each $x\not\in\oo{F}$. Pick $x_0\in K\setminus \overline{F}$ and choose a norm-one function $h\in C_0(K)$ such that $h(x_0)=1$ and $h(x)=0$ for all $x\in \overline{F}$. Let $f = h - g$. Then $\|f\|=1$ and $\|f-f_n\|>1$ for all $n\in \mathbb{N}$; a contradiction with the maximality of $\{f_n\colon n\in \mathbb N\}$.\end{proof}

Unlike the compact case, if $K$ is a large enough discrete space then $\mathsf{k}(C_0(K))<\mathsf{d}(C_0(K))$ (of course, $C_0(K) = c_0(|K|)$). Indeed, note that $\mathsf d(c_0(\mathfrak{c}^+)) = \mathfrak{c}^+$, however, the following inequality holds true (which is due to P.~Koszmider):

\begin{proposition}$\mathsf k (c_0(\mathfrak{c}^+)) \leqslant \mathfrak{c}.$  \end{proposition}
\begin{proof}Assume that there exists a (1+)-separated family of unit vectors $\EuScript A\subset c_0(\mathfrak{c}^+)$ with $|\EuScript A| = \mathfrak{c}^+$. Consider the family of supports of elements of $\EuScript A$, that is, 
$$\mathscr{A}=\big\{ \{\alpha<\mathfrak{c}^+\colon f(\alpha)\neq 0\}\colon f\in \EuScript B\big\}.$$
Clearly, $|\mathscr{A}|=\mathfrak{c}^+$. By the generalised $\Delta$-system lemma (\cite[Theorem 1.6]{kunen}), there exist a subfamily $\mathscr{B}\subseteq \mathscr{A}$ with $|\mathscr{B}|=\mathfrak{c}^+$ and a countable set $\Delta$ such that $A\cap B = \Delta$ for all distinct $A,B\in \mathscr{B}$. Given $f,g\in \EuScript A$ with supports in $\mathscr{B}$, $|f(\alpha)-g(\alpha)|>1$ implies that $\alpha \in \Delta$. This is a contradiction because there are \emph{only} continuum many real-valued functions on $\Delta$ yet we have $\mathfrak{c}^+$ different functions in $\EuScript A$ with supports in $\mathscr{B}$.\end{proof}

\subsection*{Acknowledgements} We are indebted to Marek C\'{u}th and Yanqi Qiu for careful reading and detecting certain slips in the preliminary version of the manuscript. We wish to thank M.~C\'{u}th also for his most valuable remarks concerning Theorem A(iii) and his suggestion of extending this assertion to the class of WLD spaces. We then thank the referee for the~detailed and helpful report.

\end{document}